\documentclass[a4j,10pt,onecolumn,oneside,notitlepage,openany,final]{report}
\usepackage{amsmath}
\usepackage{amsthm}
\usepackage{amsfonts}
\usepackage{amssymb}
\usepackage{bm}
\usepackage{graphicx}
\usepackage{longtable}

\newtheorem{theorem}[]{Theorem} 
\newtheorem{Conj}[theorem]{Conjecture}

\newtheorem{lemma}[theorem]{Lemma}
\newtheorem{corollary}[theorem]{Corollary}
\newtheorem{proposition}[theorem]{Proposition}
\newtheorem{remark}[theorem]{Remark}

\begin{document}
\title{\textbf{\mathversion{bold} 
On a conjecture of Moreno-Fr\'{i}as and Rosales for numerical semigroups}}
\author{Masahiro WATARI}
\date{}
\maketitle

\begin{abstract}
The present paper addresses a semimodule counting conjecture of Moreno-Fr\'{i}as and Rosales for numerical semigroups. 
Applying Pfister and Steenbrink's Theory for punctual Hilbert schemes of curve singularities, we show that this conjecture is true for any numerical semigroup.
\end{abstract}

\begin{center}
\textit{Dedicated to Professor Fumio Sakai on the occasion of his 75th birthday}
\end{center}
\section{Introduction} 

Let $\mathbb{N}$ be the set of natural numbers. 
We put $\mathbb{N}_0:=\mathbb{N}\cup\{0\}$. 
A nonempty subset $S$ of  $\mathbb{N}_0$ is called a \emph{numerical semigroup}, if it satisfies the following three condtions:
(1) $S$ is closed under addition, (2) $0\in S$, (3) $\# (\mathbb{N}_0\setminus S) < \infty$. 
For a numerical semigroup $S$, we define its \emph{Frobenius number} $F(S)$ to be the biggest gap of $S$ 
(i.e. $F(S):=\max\{n\in\mathbb{N}|\,n\notin S\}$).
We also define its
$\emph{conductor}$ and \emph{genus} to be
$c:=c(S):=F(S)+1$ and $g(S):=\# (\mathbb{N}_0\setminus S)$ respectively.

For a nonempty subset $A$ of $\mathbb{N}_0$, we set 
$\langle A\rangle:=\{\Sigma_{i=1}^n\lambda_ia_i|\,n\in \mathbb{N},\, a_i\in A, \lambda_i\in \mathbb{N}_0\}$.   
This is the  submonoid of $(\mathbb{N}_0,+)$ generated by $A$. 
In particular, we  write $\langle A\rangle=\langle a_1, \ldots, a_p \rangle$ for $A=\{a_1, \ldots, a_p\}$. 
It is known that $\langle A\rangle$ is a numerical semigroup if and only if $\gcd(A)=1$ (see \cite{Rosales and Garcia}, Lemma\,2.1).
We call $A$ a \emph{system of generators} of a numerical semigroup $S$, if $S= \langle A\rangle$. 
In the present paper, we always assume that the system of generators of $S$ is \emph{minimal} 
(i.e. $S\neq \langle B\rangle$ holds for any proper subset $B\subsetneq A$).  

For a numerical semigroup $S$, a subset $\Delta\subseteq S$ is called a \emph{$S$-semimodule}, if the relation $\Delta+S\subseteq \Delta$ holds.   
In the context of semigroup theory, $S$-semimodules are called ideals of $S$ and have been studied extensively 
(cf. \cite{AJ, R,RGG}).
We denoted by $\mathrm{Mod}(S)$ the set of all $S$-semimodules. 
For $\Delta\in\mathrm{Mod}(S)$, its \emph{codimension} is defined by $\mathrm{codim}(\Delta):=\#(S\setminus \Delta)$.
For $r\in\mathbb{N}_0$, we set $\mathrm{Mod}_r(S):=\{\Delta\in\mathrm{Mod}(S)\,|\,\mathrm{codim}(\Delta)=r\}$. 
We write $\Delta=\langle \alpha_1,\cdots,\alpha_q\rangle_{S}$ for a $S$-semimodule $\Delta$ which is generated by $\{\alpha_1,\cdots,\alpha_q\}$ 
(i.e $\Delta=\cup_{i=1}^q\{\alpha_i+S\}$). 
Similar to numerical semigroups, we call the set $\{\alpha_1,\cdots,\alpha_q\}$ a \emph{system of generators} of $\Delta$. 
A system of generators of $\Delta$ is said to be \emph{minimal}, if  the relations $\Delta\supsetneq \cup_{i=1,i\neq l}^q\{\alpha_i+S\}$ hold for all $l\in\{1,\ldots,q\}$. 
Note that if $\Delta$ is a $S$-semimodule, then $\widetilde{\Delta}:=\Delta\cup \{0\}$ is a numerical semigroup. 
We call $\widetilde{\Delta}$ the \emph{numerical semigroup associated with} $\Delta$.  
For $r\in\mathbb{N}_0$, we put $\mathcal{J}(S,r):=\{\widetilde{\Delta}\,|\,\Delta\in \mathrm{Mod}(S) \text{ and }g(\widetilde{\Delta})-g(S)=r\}$. 
The following conjecture was considered by Moreno-Fr\'{i}as and Rosales in \cite{Moreno-Frias and Rosales}.

\begin{Conj}{\rm (\cite{Moreno-Frias and Rosales})}\label{Main conjecture}
Let $S$ be a numerical semigroup. 
There exists a positive integer $n_S$ such that if relations $0\le a< b \le n_S$ hold, then  $\# \mathcal{J}(S,a)\le \# \mathcal{J}(S,b)$ and $\# \mathcal{J}(S,n)= \# \mathcal{J}(S,n_S)$ hold for any positive integer $n$ with $n\ge n_S$.
\end{Conj}

Moreno-Fr\'{i}as and Rosales proved this conjecture for the numerical semigroup generated by $2$ and $2k+1$ in \cite{Moreno-Frias and Rosales}. 
Chavan also argued this conjecture in \cite{Chaven}. 
In the present paper, we prove the following theorem:

\begin{theorem}\label{Main theorem}
For any numerical semigroup $S$, Conjecture\,{\rm \ref{Main conjecture}} is true with $n_S=F(S)$. 
\end{theorem}

This paper is laid out as follows:
In Section\,\ref{preliminaries}, we briefly recall Pfister and Steenbrink's theory for punctual Hilbert schemes of curve singularities. 
Applying the properties of punctual Hilbert schemes, 
we prove Theorem\,\ref{Main theorem} in Section\,\ref{proof}.

\section{Preliminaries}\label{preliminaries}
In this section, we briefly summarize basic notions and known facts about punctual Hilbert schemes for irreducible curve singularities. 
Our main reference is \cite{PS}. 
Let $\mathbb{C}$ be the field of complex numbers. 
For a reduced and irreducible singular curve $X$ over $\mathbb{C}$ and its singular point $o$, we refer to the pair $(X,o)$ as an \emph{irreducible curve singularity}. 
Let $\mathcal{O}=\mathcal{O}_{X,o}$ be the complete local ring of the irreducible curve singularity $(X,o)$. 
For an irreducible curve singularity $(X,o)$, we define the \emph{punctual Hilbert scheme} $\mathrm{Hilb}^r (X,o)$ of degree $r$ to be the set of length $r$ subschemes of $X$ supported on $o$.
Note that the structure of $\mathrm{Hilb}^r (X,o)$ depends only on the completed local ring $\mathcal{O}$. 
There is a natural identification
$$
\mathrm{Hilb}^r (X,o)=\{I\,|\,\text{$I$ is an ideal of $\mathcal{O}$ with $\dim_{\mathbb{C}}\mathcal{O}/I=r$}\}.
$$ 
Below, we always consider an irreducible curve singularity $(X,o)$ with $\mathcal{O}=\mathbb{C}[[t^{a_1},\cdots,t^{a_p}]]$ for some positive integers $a_1,\ldots,a_p$. 
Such a curve singularity is called a \emph{monomial curve singularity}.  
Without loss of generality, we may assume that $\gcd(a_1,\ldots,a_p)=1$. 
Note that the normalization $\overline{\mathcal{O}}$ of $\mathcal{O}$ is isomorphic to $\mathbb{C}[[t]]$. 
The set $\{\text{ord}_t(f)\,|\,f\in \mathcal{O}\}$ is a numerical semigroup. 
It is called the \emph{semigroup} of $\mathcal{O}$. 
The positive integer $\delta:=\dim _{\mathbb{C}}(\overline{\mathcal{O}}/\mathcal{O})$ is called the \emph{$\delta$-invariant} of $\mathcal{O}$.  

For a positive integer  $a$, we denote by $(t^a)$ the ideal of $\overline{\mathcal{O}}$ generated by $t^a$. 
Let $\mathrm{Gr}\left(\delta,\overline{\mathcal{O}}/(t^{2\delta})\right)$ be the Grassmannian which consists of $\delta$-dimensional linear subspaces of $\overline{\mathcal{O}}/(t^{2\delta})$.
For $M\in\mathrm{Gr}\left(\delta,\overline{\mathcal{O}}/(t^{2\delta})\right)$, we define a multiplication by $\overline{\mathcal{O}}\times M\ni (f,m+(t^{2\delta}))\mapsto fm+(t^{2\delta})\in M$. 
Set
$$
J:=\left\{M\in \mathrm{Gr}\left(\delta,\overline{\mathcal{O}}/(t^{2\delta})\right)\big|\,\text{$M$ is an $\mathcal{O}$-submodule w.r.t. the above multiplication}\right\}.
$$
This set $J$ is called the \emph{Jacobian factor} of the singularity $(X,o)$. 
It was introduced by Rego in \cite{Re}. 
Piontkowski  \cite{P} described the homology of the Jacobian factors for the irreducible plane curve singularities whose Puiseux characteristics are $(p,q)$, $(4,2q,s)$, $(6,8,s)$ and $(6,10,s)$.

In \cite{PS}, Pfister and Steenbrink defined a map 
\begin{equation}\label{phai}
\varphi_r: \mathrm{Hilb}^r (X,o)\rightarrow J\subset \mathrm{Gr}(\delta, \overline{\mathcal{O}}/(t^{2\delta}))
\end{equation}
by $\varphi _r (I)=t^{-r}I/(t^{2\delta})$. 
We call $\varphi _r$ the \emph{$\delta$-normalized embedding}. 
It has the following properties:

\begin{proposition}[\cite{PS}, Theorem\,3]\label{thm3}
Let $S$ be the semigroup of $\mathcal{O}$.  
The $\delta$-normalized embedding $\varphi_r$  is injective for $r\in\mathbb{N}_0$. 
Furthermore, it is bijective for $r\ge c$. 
The image $\varphi _r (\mathrm{Hilb}^r (X,o))$  is Zariski closed in $J$. 
\end{proposition}

\begin{corollary}\label{stable}
If $r\ge c$, then we have  $\varphi_r(\mathrm{Hilb}^r (X,o))\cong\varphi_c(\mathrm{Hilb}^{c}(X,o))\cong J$. 
\end{corollary}

\begin{remark}
The stability in Corollary\,{\rm \ref{stable}} was generalized to the case of multi--branch curve singularities in {\rm \cite{BRV}}.  
\end{remark}

Let $S$ be the semigroup of $\mathcal{O}$.  
For any ideal $I$ of $\mathcal{O}$, the order set $S(I):=\{\text{ord}_t(f)|\,f\in I\}$ is a $S$-semimodule. 
It is known that the punctual Hilbert scheme $\mathrm{Hilb}^r (X,o)$ is decomposed in terms of $S$-semimodules $\Delta$.

\begin{proposition}[\cite{SW1}, Proposition\,6]\label{decomposition lemma}
For the semigroup $S$ of $\mathcal{O}$, we have
\begin{equation}\label{decomposition}
\mathrm{Hilb}^r (X,o)=\bigsqcup_{\Delta\in\mathrm{Mod}_{r}(S)} \mathcal{I}(\Delta)
\end{equation}
where $\mathcal{I}(\Delta):=\{I|\,\text{$I$ is an ideal of $\mathcal{O}$ with $S(I)=\Delta$}\}$. 
\end{proposition}

We call $\mathcal{I}(\Delta)$ the \emph{$\Delta$-subset} for $\Delta$. 
\begin{lemma}\label{not empty}
If $(X,o)$ is a monomial curve singularity, then any $\Delta$-subset $\mathcal{I}(\Delta)$ 
is non-empty. 
\end{lemma}

\begin{proof}
Let $(X,o)$ be a monomial curve singularity with its local ring $\mathcal{O}=\mathbb{C}[[t^{a_1},\cdots,t^{a_p}]]$. 
The numerical semigroup $S$ of $(X,o)$ is given by $S=\langle a_1,\cdots,a_p\rangle$. 
Let $\Delta=\langle \alpha_1,\cdots,\alpha_q\rangle_S$ be an element of $\mathrm{Mod}(S)$. 
For a generator $\alpha_i$ of $\Delta$, there exist $\lambda_{ij}\in\mathbb{N}_0$ $(i=1,\ldots,q)$
such that $\alpha_i=\sum_{j=1}^p\lambda_{ij}a_j$. 
Since the ideal $I$ of $\mathcal{O}$  generated by
$
t^{\sum_{j=1}^p \lambda_{ij}a_j}$ $(i=1,\ldots,q)
$
satisfies $S(I)=\Delta$. 
Hence we have $\mathcal{I}(\Delta)\neq \emptyset$. 
\end{proof}

\begin{remark}
In general, Lemma\,{\rm \ref{not empty}} is not true. 
For the plane curve singularity with the Puiseux characteristic $(4,2q,s)$, 
there exist $\Delta$-subsets that are empty $($see Theorem\,$13$ in {\rm \cite{P})}.
\end{remark}

\section{Proof of the main theorem}\label{proof}

In this section, we first prove some properties of numerical semigroups and their semimodules. 
Applying them, we finally prove Theorem\,\ref{Main thm}. 

\begin{lemma}\label{proposition r r+1}
Let $\Delta=\langle \alpha_1,\cdots,\alpha_q\rangle_{S}$  be a $S$-semimodule.  
If  $\Delta$ is an element of $\mathrm{Mod_r(S)}$ $(r\in \mathbb{N}_0)$, then 
$\Delta\setminus \{\alpha_i\}$ is an element of  $\mathrm{Mod_{r+1}(S)}$. 
Conversely, if  $\Delta$ is an element of  $\mathrm{Mod_{r+1}(S)}$, then 
 there exist at least one element $s$ of  $S\setminus \Delta$ such that $\Delta\cup \{s \}$ is an element of $\mathrm{Mod_r(S)}$.
\end{lemma}

\begin{proof}
If $\Delta=\langle \alpha_1,\cdots,\alpha_q\rangle_{S}$ is an element of $\mathrm{Mod_r(S)}$, then it is easy to show that $\Delta\setminus \{\alpha_i\}$ is a  $S$-semimodule for any $\alpha_i$.  
Furthermore, since $\mathrm{codim} (\Delta\setminus \{\alpha_i\})=\#\{S \setminus (\Delta\setminus \{\alpha_i\})\}=r+1$. 
we have $\Delta\setminus \{\alpha_i\}\in \mathrm{Mod_{r+1}(S)}$. 
For  $\Delta\in\mathrm{Mod}_{r+1}(S)$, let $\alpha:=\max\{s\in S\,|\,s\notin \Delta\}$. 
It is obvious that $\Delta\cup \{\alpha\}$ belongs to $\mathrm{Mod}_r(S)$. 
Note that $\alpha$ is one of the generators of the new $S$-semimodule $\Delta\cup \{\alpha\}$. 
\end{proof}


\begin{proposition}\label{monotonically increase}
For $r\in\mathbb{N}_0$, the relation $\# \mathrm{Mod}_r(S)\le \# \mathrm{Mod}_{r+1}(S)$ holds.
\end{proposition}

\begin{proof}
Note that an element of $\mathrm{Mod}_{r}(S)$ yields  some elements of $\mathrm{Mod}_{r+1}(S)$ and an element of $\mathrm{Mod}_{r+1}(S)$ is yielded from at least one element of $\mathrm{Mod}_{r}(S)$ by Lemma\,{\rm \ref{proposition r r+1}}. 
Write $\mathrm{Mod}_{r}(S)=\{\Delta_{r1},\ldots,\Delta_{rN_r}\}$ where $\Delta_{ri}=\langle \alpha_{i1}\ldots\alpha_{iq_i}\rangle_S$. 
We assume that  any system of generators $\{\alpha_{i1}\ldots\alpha_{iq_i}\}$ is minimal. 
We infer from Lemma\,{\rm \ref{proposition r r+1}} that

\begin{equation}\label{r+1}
\mathrm{Mod}_{r+1}(S)= \bigcup_{i=1}^{N_r}\{\Delta_{ri}\setminus \{\alpha_{ij}\}|\, j=1,\dots,q_i\}
\end{equation}
where $\{\Delta_{ri}\setminus \{\alpha_{ij}\}|\, j=1,\dots,q_i\}\neq\emptyset$ for any $i$. 

To prove our assertion, it is enough to show the following fact: 
for any two distinct elements of $\mathrm{Mod}_{r}(S)$,  
$\Delta_{rk}=\langle \alpha_{k1},\cdots,\alpha_{kq_k}\rangle_{S}$ and  
$\Delta_{rl}=\langle \alpha_{l1},\cdots,\alpha_{lq_l}\rangle_{S}$ where $1\le k<l\le N_r$, we have
\begin{equation}\label{cardinality}
\{\Delta_{rk}\setminus \{\alpha_{kj}\}|\, j=1,\dots,q_k\}\neq\{\Delta_{rl}\setminus \{\alpha_{lj}\}|\, j=1,\dots,q_l\}.
\end{equation}
If the relation\,(\ref{cardinality}) holds, then, by mathematical induction, we can show that
\begin{equation*}
\#\mathrm{Mod}_{r+1}(S)=\# \bigcup_{i=1}^{N_r}\{\Delta_{ri}\setminus \{\alpha_{ij}\}|\, j=1,\dots,q_i\} \ge N_r=\#\mathrm{Mod}_{r}(S).
\end{equation*}

To prove (\ref{cardinality}), we consider the following two cases:

\noindent
(Case 1): $\{ \alpha_{k1},\cdots,\alpha_{kq_k}\}\cap\{ \alpha_{l1},\cdots,\alpha_{lq_l}\}\neq \emptyset$. 
 In this case,  it is obvious that $\Delta_{rk}\setminus\{\alpha\}\neq\Delta_{rl}\setminus\{\alpha\}$
for any $\alpha\in\{ \alpha_{k1},\cdots,\alpha_{kq_k}\}\cap\{ \alpha_{l1},\cdots,\alpha_{lq_l}\}$. 
Furthermore, $\Delta_{rk}\setminus\{\alpha\}$ and $\Delta_{rl}\setminus\{\alpha\}$ are elements of $\mathrm{Mod}_{r+1}(S)$ by Lemma\,{\rm \ref{proposition r r+1}}. 
So the relation\,(\ref{cardinality}) holds.  

\noindent
(Case 2): $\{ \alpha_{k1},\cdots,\alpha_{kq_k}\}\cap\{ \alpha_{l1},\cdots,\alpha_{lq_l}\}= \emptyset$. 
We see that $\Delta_{rk}\setminus\{\alpha_{km}\}\neq\Delta_{rl}\setminus\{\alpha_{ln}\}$ for any $m,n$. 
Indeed, if $\Delta_{rk}\setminus\{\alpha_{km}\}=\Delta_{rl}\setminus\{\alpha_{ln}\}$ holds for some $m$ and $n$, 
then we have $\Delta_{rk}=(\Delta_{rl}\setminus\{\alpha_{ln}\})\cup\{\alpha_{km}\}$. 
For $\alpha_{ku}$ with $u\neq m$ and $1\le u \le q_k$, there exist $\alpha_{lv}$ with $v\neq n$ and $1\le v \le q_l$ and 
$s_v\in S$ such that $\alpha_{ku}=\alpha_{lv}+s_v$.
Similarliy, for this $\alpha_{lv}$, there exist $\alpha_{kw}$ with $w\neq m$ and $1\le w \le q_k$ and 
$s_w\in S$ such that $\alpha_{lv}=\alpha_{kw}+s_w$. 
So we obtain $\alpha_{ku}=\alpha_{kw}+s_w+s_v$.
However, it contradicts the minimality of the system of generators for $\Delta_{rk}$. 
Again the relation (\ref{cardinality}) holds.

\noindent
Our assertion is proved. 
\end{proof}

\noindent
\textbf{Proof of Theorem\,\ref{Main theorem}.} 
Note that if $r\in \mathbb{N}$, then $g(\widetilde{\Delta})-g(S)=\mathrm{codim}(\Delta)-1$ holds 
for any $\Delta\in \mathrm{Mod}_r(S)$.  
So $\mathcal{J}(S,r)$ can be expressed as $\mathcal{J}(S,r)=\{\widetilde{\Delta}|\,\Delta\in\mathrm{Mod}_{r+1}(S)\}$. 
This fact implies that $\# \mathcal{J}(S,r)=\#\mathrm{Mod}_{r+1}(S)$ for $r\in \mathbb{N}_0$. 
It follows from Proposition\,\ref{monotonically increase} that 
$\#\mathcal{J}(S,r)=\# \mathrm{Mod}_{r+1}(S)\le \# \mathrm{Mod}_{r+2}(S)=\#\mathcal{J}(S,r+1)$ for $r\in\mathbb{N}_0$. 
Namely, the sequence $\{\#\mathcal{J}(S,r)\}_{r\in \mathbb{N}_0}$ weakly increases. 

Next, we show that 
$\{\#\mathcal{J}(S,r)\}_{r\in \mathbb{N}_0}$ eventually stabilizes. 
For a numerical semigroup $S$, we set $\mathbb{C}[[S]]:=\mathbb{C}[[t^i\,|\,i\in S]]$ and 
consider the monomial curve singularity $(X,o)$ whose local ring $\mathcal{O}$ is $\mathbb{C}[[S]]$. 
It is obvious that the semigroup of $\mathcal{O}$ is $S$. 

If $r\ge c$, then the relation
\begin{equation}\label{equality}
\varphi_r(\mathrm{Hilb}^r(X,o))=J=\varphi_{c}(\mathrm{Hilb}^{c}(X,o))
\end{equation}
holds by Corollary\,\ref{stable}.  
Furthermore, writing $\mathrm{Mod}_{i}(S)=\{\Delta_{i1},\ldots,\Delta_{iN_i}\}$ where $i=r,c$, 
the both sides of the relation (\ref{equality}) can be rewritten as 
\begin{equation}\label{equality2}
\bigsqcup_{j=1}^{N_r} \varphi_r\big(\mathcal{I}(\Delta_{r,j})\big)
=\bigsqcup_{j=1}^{N_c} \varphi_{c}\big(\mathcal{I}(\Delta_{c,j}))\big) 
\end{equation}
by  Proposition\,\ref{decomposition lemma}. 
Here any  $\mathcal{I}(\Delta_{r,j})$ and $\mathcal{I}(\Delta_{r,j})$ are non-empty by Lemma\,\ref{not empty}. 
Note that, 
for an element $\Delta_{r,k}$ of $\mathrm{Mod}_{r}(S)$, 
there exists a unique element $\Delta_{c,l}$ of $\mathrm{Mod}_{c}(S)$
such that  the relation $\varphi_r\big(\mathcal{I}(\Delta_{r,k})\big)=\varphi_c\big(\mathcal{I}(\Delta_{c,l})\big)$
holds in {\rm(\ref{equality2})}. 
Indeed, since $\varphi_r$ and $\varphi_c$ are bijective by Propostion\,\ref{thm3}, for an ideal $I\in \mathcal{I}(\Delta_{r,k})$,
there exists  a unique $\Delta_{c,l}\in\mathrm{Mod}_{c}(S)$ and a unique element $I'\in \mathcal{I}(\Delta_{c,l})$
such that $\varphi_r(I)=\varphi_c(I')\in\varphi_c\big(\mathcal{I}(\Delta_{c,l})\big)$. 
This fact implies that $S$-modules of $\mathrm{Mod}_{r}(S)$ and $\mathrm{Mod}_{c}(S)$ are in one to one correcepondence. 
That is, $\#\mathrm{Mod}_{r}(S)=\# \mathrm{Mod}_{c}(S)$ holds for $r\ge c$. 
It follows that $\#\mathcal{J}(S,r-1)=\# \mathrm{Mod}_{r}(S)=\# \mathrm{Mod}_{c}(S)= \#\mathcal{J}(S,c-1)= \#\mathcal{J}(S,F(S))$  for $r\ge c$. 
$\square$

\noindent
\small 
Masahiro Watari\\
 University of Kuala Lumpur, Malaysia France Institute\\
Japanese Collaboration Program\\
Section 14, Jalan Damai, Seksyen 14, 43650\\
Bandar Baru Bangi, Selengor, Malaysia.\\
E-mail:masahiro@unikl.edu.my
\end{document}